\documentclass[reqno,makeidx,12pt]{amsart}
\usepackage[utf8]{inputenc}

\usepackage[usenames,dvipsnames]{color}
\usepackage{graphicx}
\usepackage{latexsym}
\usepackage{amstext}
\usepackage {amsmath}
\usepackage {amsfonts}
\usepackage {amssymb}
\usepackage {amsthm}
\usepackage {bbm}
\usepackage{enumerate}
\usepackage{array}
\usepackage{comment}
\usepackage{xspace}
\usepackage[bookmarks=true]{hyperref}
\usepackage{enumerate}
\ifx\cs\documentclass \footheight 12pt \fi \footskip 30pt
\allowdisplaybreaks

\DeclareMathOperator{\dive}{div}

\DeclareMathOperator{\tr}{tr}
\newcommand{\dft}{d}

\newcommand{\eps}{\varepsilon}

\newcommand{\R}{\ensuremath{\mathbb{R}}}
\newcommand{\Rn}{\ensuremath{{\mathbb{R}^n}}}
\newcommand{\N}{\ensuremath{\mathbb{N}}}
\newcommand{\Ex}{\ensuremath{\mathbb{E}}} 
\newcommand{\LL}{\ensuremath{\mathcal{L}}}

\newcommand{\Ha}{\ensuremath{\mathcal{H}}}

\def\Per{\mathcal{P}}
\def\Prob{\ensuremath{{\mathbb{P}}}}

\newcommand{\ue}{u_\varepsilon} 
\newcommand{\En}{E_\varepsilon} 
\newcommand{\F}{\mathcal{F}} 
\newcommand{\ar}{\mathcal{A}} 
\theoremstyle{plain}
\numberwithin{equation}{section}
\newtheorem{lemma}{Lemma}[section]
\newtheorem{theorem}[lemma]{Theorem}
\newtheorem{proposition}[lemma]{Proposition}

\theoremstyle{definition}
\newtheorem{remark}[lemma]{Remark}

\begin{document}
\title[Stochastic Allen--Cahn equation]{Tightness for a stochastic Allen--Cahn equation}

\author{Matthias R{\"o}ger}
\address{Matthias R\"oger, Technische Universit\"at Dortmund,
Faculty of Mathematics,
Vogelpothsweg 87,
44227 Dortmund, Germany}

\author{Hendrik Weber}
\address{Hendrik Weber, Mathematics Institute, 
University of Warwick, Coventry CV4 7AL - United Kingdom}

\email{matthias.roeger@tu-dortmund.de, hendrik.weber@warwick.ac.uk}



\thanks{This work was partially supported by the DFG Forschergruppe FOR 718 \emph{Analysis and Stochastics in Complex Physical Systems}}
\date{\today}

\begin{abstract}
We study an Allen--Cahn equation perturbed by a multiplicative stochastic 
noise that is white in time and correlated in space. Formally this equation
approximates a stochastically forced mean curvature flow. We derive a uniform bound for the diffuse surface area,
prove the tightness of solutions in the sharp interface limit, and show the convergence to phase-indicator functions. 

\end{abstract}

\maketitle
\section{Introduction}
\label{sec:intro}
The Allen--Cahn equation
\begin{gather}
\eps \partial_t u_\eps \,=\, \eps\Delta u_\eps -\frac{1}{\eps}F'(u_\eps)
\label{eq:AC}
\end{gather}
is an important prototype for phase separation processes in melts or alloys that is of
fundamental interest both for theory and applications. It describes an
evolution of non-conserved phases  driven by the surface area reduction of their common interface.
The Allen--Cahn equation is a \emph{diffuse interface} model, {\itshape i.e.}
phases are indicated by smooth fields, assuming a partial mixing of the
phases. It
is well-known \cite{MoSc90,EvSS92,Ilma93}  that in the \emph{sharp interface limit} $\eps\to
0$ solutions of the Allen--Cahn equation converge to an evolution of
hypersurfaces $(\Gamma_t)_{t\in 
(0,T)}$ by \emph{mean curvature flow} (MCF)
\begin{gather}
v(t,\cdot)\,=\, H(t,\cdot), \label{eq:fmc}
\end{gather}
where $v$ describes the velocity vector of the evolution and $H(t,\cdot)$
denotes the mean curvature vector of $\Gamma_t$.

Our goal is to introduce a stochastic perturbation of the Allen--Cahn
equation that formally approximates a stochastic mean curvature type
flow
\begin{gather}
	v(t,\cdot) \,=\, H(t,\cdot) + X(t,\cdot), \label{eq:smcf-formal}
\end{gather}
where $X$ now is a random vector-field in the ambient space.
More specifically what we are considering is the following Stratonovich stochastic
partial differential equation (SPDE):
\begin{gather}
  \dft\ue\,=\, \Big( \Delta \ue -\frac{1}{\eps^2}F'(\ue)\Big) dt +
  \nabla\ue\cdot X(x,\circ dt), \label{eq:sac}
\end{gather}
where $X$ is a vectorfield valued Brownian motion. A particular case of such 
a Brownian motion is 
\begin{gather}
 X(t,x)=X^{0}(x)t+\sum_{k=1}^N X^{k}(x)  B_k(t) ,
\end{gather}
where the $X^{k}$ are fixed vectorfields and $W_k$ are independent standard 
Brownian motions. In this case \eqref{eq:sac} reduces to the Stratonovich
SPDE
\begin{gather}
  \dft\ue\,=\, \Big( \Delta \ue -\frac{1}{\eps^2}F'(\ue) +
 \nabla\ue\cdot X^{0} \Big) dt+ \sum_{k=1}^N \nabla\ue\cdot X^{k} \circ dB_k(t).  \label{eq:sac1}
\end{gather}
Our setting is more general as it allows for infinite sums of Brownian motions. 
See below for a more detailed discussion.
We complement \eqref{eq:sac} by deterministic 
initial and zero Neumann-boundary data, 
\begin{align}
u_\eps(0,\cdot)\,&=\, u_\eps^0\quad\text{ in } U, 
\label{eq:init}
\\
\nabla u_\eps\cdot\nu_U\,&=\, 0\quad\text{ on }(0,T)\times
\partial U,
\label{eq:bdry}
\end{align} 
where $\nu_U$ denotes the outer normal vector to $\partial U$.

Our main result is the tightness of the solutions 
$(u_\eps)_{\eps>0}$ of \eqref{eq:sac} and the convergence to an
evolution of (random) phase indicator functions 
$u(t,\cdot)\in BV(U)$.  In particular we prove 
a uniform control (in $\eps>0$) of the diffuse surface area of $u_\eps$.

In the next sections we briefly review the analysis of the
deterministic Allen--Cahn equation and report on stochastic
extensions. In Section \ref{sec:prelim} we state our main assumptions
and recall some notations for stochastic flows.
Our main results are stated in Section \ref{sec:res}. In Section
\ref{sec:ExUn} we prove an existence result for \eqref{eq:sac}. The main estimates for the diffuse surface area and the
tightness of solutions are proved in Section \ref{sec:Tightness}. In the last section we present a discussion of the relation to a stochastically perturbed mean curvature flow by a formal passage to the sharp interface limit in a localized surface energy equality.
\subsection*{Acknowledgment}
This work was partially supported by the DFG Forschergruppe 718 \emph{Analysis and Stochastics in Complex Physical Systems}.
\section{Background}
\label{sec:back}
\subsection{Deterministic sharp interface limit}
\label{subsec:det}
As many other diffuse interface models the Allen--Cahn equation
\eqref{eq:AC} is based on the \emph{Van~der~Waals--Cahn--Hilliard} energy
\begin{gather}
E_{\varepsilon}(u_\eps)\,:=\,\int_{U}\Big(\frac{\varepsilon}{2}|\nabla u_\eps|^2 
+\frac{1}{\varepsilon}F(u_\eps)\Big)\, dx\qquad\text{ for }u_\eps:U\to\R. \label{eq:def-E_eps}
\end{gather}
The energy $E_\eps$
favors a decomposition of the spatial domain $U$ into two regions (phases) where
$u_\eps\approx -1$ and $ u_\eps\approx 1$, separated by a transition layer (diffuse
interface) with a thickness of order $\eps$. Modica and Mortola
\cite{MoMo77,Modi87} proved that $E_\eps$  
Gamma-converges (with respect to $L^1$-convergence) to a constant
multiple of the perimeter functional $\Per$, restricted to phase indicator
functions,
\begin{gather*}
E_\eps\,\to\, c_0\Per,\qquad
\Per(u)\,:=\, 
\begin{cases}
  \frac{1}{2}\int_U \,d|\nabla u| &\text{ if } u\in BV(U,\{-1,1\}),\\
  \infty &\text{ otherwise.}
\end{cases}
\end{gather*}
$\Per$ measures the surface-area of the phase boundary
$\partial^*\{u=1\}\cap U$. In this sense $E_\eps$ describes a diffuse
approximation of the surface-area functional.

The Allen--Cahn equation \eqref{eq:AC} in fact is the (accelerated)
$L^2$-gradient flow of $E_\eps$. It is
proved in different formulations \cite{MoSc90,EvSS92,Ilma93} that 
\eqref{eq:AC} converges to motion by mean curvature. Since mean
curvature flow in general allows for the formation of singularities 
in finite time it is necessary to consider suitable generalized
formulations of \eqref{eq:fmc}, as for example in the sense of
viscosity solutions
\cite{BaSS93,ChGG91,CrIL92,EvSS92,EvSp91}, De 
Giorgi's barriers \cite{BeNo97,BePa95,ChNo08,DiLN01}, or
geometric measure theory. The first approaches rely on the
maximum principle, the latter was pioneered by Brakke \cite{Brak78} and is
based on the \emph{localized} energy equality
\begin{align}
&\frac{d}{dt}\int_{\Gamma_t}\eta(x)\,d\Ha^{n-1}(x) \notag\\
&\, =\,
\int_{\Gamma_t}\nabla\eta(t,x)\cdot V(t,x)\,d\Ha^{n-1}(x) -
 \int_{\Gamma_t}H(t,x)\cdot V(t,x)\eta(t,x)\,d\Ha^{n-1}(x) \label{eq:loc-en-mcf}
\end{align}
that holds for arbitrary $\eta\in C^1_c(U)$ and for any classical solutions
$(\Gamma_t)_{t\in (0,T)}$ of mean curvature flow. Ilmanen \cite{Ilma93} proved the
convergence of the Allen--Cahn equation to mean curvature flow in the
formulation of Brakke, using a diffuse analog of the (localized) energy
equality \eqref{eq:loc-en-mcf}. By similar
methods Mugnai and the first author \cite{MuRoe11} proved the convergence of
(deterministically) perturbed Allen--Cahn equations.

One of the key results of the present paper is an energy inequality for
the stochastic Allen--Cahn equation \eqref{eq:sac}. By It\^{o}'s 
formula the stochastic drift produces some extra terms in the
time-derivative of the diffuse surface energy $E_\eps(u_\eps)$. These
`bad' terms are exactly compensated by the additional terms
in \eqref{eq:sac} which are hidden in the Stratonovich formalism. 
\subsection{Stochastic perturbations of the Allen--Cahn equation and MCF}
\label{subsec:stoch-pert}
Additive perturbations of the Allen--Cahn equation were studied in the one-dimensional
case in \cite{Fu95,BdMP95} and in the higher-dimensional case in \cite{Fu99,We08,LioSou98}.
Note that perturbation results such as \cite{MuRoe11} do not apply to the stochastic case as one 
typically perturbs with a white noise i.e. the time-derivative of a
$C^\alpha$ function for $\alpha< \frac{1}{2}$,
which is not covered by most techniques.

In the one-dimensional case the equation was studied with an additive space-time white 
noise and at least for the case where the interface consists of a single kink the sharp interface 
limit was described rigorously \cite{Fu95,BdMP95}. In higher dimensions
the picture is much less complete. For instance, the
Allen--Cahn equation with space-time white noise is in general not well-posed:
the noise term is so rough that for $n \geq 2$ solutions to the
stochastic heat equation attain values only in Sobolev spaces 
of negative order, and on such spaces the nonlinear potential can a priori not be defined. This existence 
problem can be avoided if one introduces spatial correlations as we do in \eqref{eq:sac}. In all of
the above papers conditions on the stochastic perturbations are much more restrictive than
in our approach. In fact it is always
assumed that the noise is constant in space and smoothened in time with a correlation length that 
is coupled to the interface width $\eps$ and goes to zero for $\eps \downarrow 0$. All of these papers rely 
on a construction of the limit dynamics by different means and then an explicit construction of sub- and 
supersolutions making use of the maximum principle. Our approach is
based only on energy estimates. On the other hand, we only prove
tightness of the approximations and do not obtain an evolution
law for limit points.

The restriction on spatially constant noise in previous papers and our
problem to derive the stochastic motion law in the limit are closely related
to the lack of existence results and generalized formulations
for stochastically forced mean curvature flow. Up to now there are only
results for spatially constant forcing \cite{DLN01, Fu99, LioSou98} or
in the case of evolution of graphs in $1+1$ dimensions \cite{ESvR09}.

Our approach is closely related to Yip's construction  \cite{Yi98} of a
time-discrete stochastically forced mean curvature flow. 
Yip follows the deterministic scheme of \cite{ATW93,LuSt95}, where for a given time step $\delta>0$ 
a sequence of sets of bounded perimeter is constructed iteratively. The heart of the
construction is the minimization of a functional that is given by the
perimeter plus a suitable distance from the previous set. Yip \cite{Yi98}
introduces randomness to this scheme by performing a stochastic flow in
between two minimization steps. For the resulting time-discrete evolution
of sets Yip proves uniform bounds (in $\delta$) for
the perimeter and shows tightness of the time-discrete solutions
with $\delta\to 0$. As in our case, a characterization of the limiting
evolution is not given. If one applies Yip's scheme to the Allen--Cahn
equation (substituting the perimeter functional by the diffuse surface area
energy and using a rescaled $L^2$-distance between phase fields) one in
fact would obtain our stochastic Allen--Cahn equation 
\eqref{eq:sac} in the limit $\delta\to 0$.

Noisy perturbations of the Allen--Cahn equation were studied from a different point of view
in \cite{KORV07}. There the authors study the action functional which
appears if one first applies Freidlin-Wentzel theory to the Allen--Cahn equation
with an additive noise that is white in time and spatially correlated, and then
formally takes the spatial correlation to zero. Then the sharp
interface limit $\varepsilon \downarrow 0$ is studied on the level of 
action functionals and a \emph{reduced action functional} as a
possible $\Gamma$-limit is derived. See \cite{WeTo07,MuRoe08} for a subsequent analysis.
\section{Assumptions and stochastic flows}
\label{sec:prelim}
\subsection{Notation and assumptions}
\label{subsec:nota}
Let $U\subset \Rn$ be an open bounded subset of $\R^n$ with smooth
boundary, let $T>0$, and set $U_T:= (0,T)\times U$. We denote by $x\in U$ and 
$t\in (0,T)$ the space- and time-variables respectively; $\nabla$
and $\Delta$ denote the spatial gradient and Laplacian. The space of functions of bounded variation in $U$ with values in $\{-1,1\}$ almost everywhere is denoted with $BV(U;\{-1,1\})$. $C^{k,\alpha}(\overline{U})$ denotes the space of $k$-times differentiable functions such that all  $k^{\text{th}}$-partial derivatives are bounded and H\"older-continuous with exponent $\alpha$.

We assume the potential $F$ to be smooth and verify the following assumptions:
\begin{equation}\label{assumptions}
\left.
\begin{split}
	&F(r) \geq 0 \qquad \text{and} \quad F(r)=0 \quad \text{iff } r=\pm 1, \\
	&F' \text{ admits exactly three zeros $\{\pm 1,0 \}$ and $F''(0) < 0$, $F''(\pm 1)>0$},\\
	&F \text{ is symmetric, } \forall r \geq 0 \quad F(r)=F(-r),\\
	& F(r) \geq C |r|^{2+\delta} \qquad \text{for some $\delta>0$ and $|r|$ sufficiently large}.
\end{split}
\right\}
\end{equation}
The standard choice for $F$ is
\begin{gather*}
F(r)\,=\,\frac{1}{4}(1-r^2)^2,
\end{gather*}
such that the nonlinearity in \eqref{eq:AC} becomes $F'(r)=-r(1-r^2)$.

Next we give some geometric meaning to $\ue$.
We define the \emph{normal direction} with respect to $\ue$ by
\begin{gather}
\nu_\eps(t,x) \,:=\, \begin{cases} 
  \frac{\nabla \ue}{|\nabla \ue|}(t,x) &\text{ if } |\nabla\ue(t,x)|\neq 0,\\
  \vec{e} &\text{ else,}
  \end{cases}
  \label{eq:def-nu}
\end{gather}
where $\vec{e}$ is an arbitrary fixed unit vector.
We define the \emph{diffuse surface area measures}
\begin{gather}
\mu_\eps^t(\eta)\,:=\, \int_U \eta\Big(\frac{\eps}{2}|\nabla \ue(t,\cdot)|^2
+\frac{1}{\eps}F(\ue(t,\cdot))\Big)\,dx \label{eq:def-mu}
\end{gather}
for $\eta \in C^0_c(U)$. We denote the \emph{diffuse mean curvature} by
\begin{gather}\label{eq:def-w_eps}
w_\eps\,:=\, -\eps\Delta\ue + \frac{1}{\eps}F'(\ue).
\end{gather}

For the initial data we assume that $u_\eps^0$ is smooth and that
\begin{gather}
E_\eps(u_\eps^0)\,\leq\, \Lambda \label{eq:ass-init} 
\end{gather}
holds for all $\eps>0$ and a fixed $\Lambda>0$. Note that by \cite[page 423]{Ilma93} 
the boundary of every open set that verifies a density bound and that can be approximated in $BV$ by 
smooth hypersurfaces  can be approximated by phase fields with uniformly bounded diffuse surface area. On the other hand \eqref{eq:ass-init} implies by
\cite{Mo87, MoMo77} that the sequence $u_\eps^0$ is compact 
in $L^1(U_T)$ and that every limit belongs to the space of phase indicator functions $BV(U, \{ \pm 1 \})$.

\subsection{Stochastic Flows}
\label{subsec:StochFlo}
Let us briefly introduce some notations for stochastic flows. We refer the reader to 
Kunita's book \cite{Kun97} Chapter 3 and Section 2,5 and 6 in Chapter 4 for further background.

Let $(\Omega,\F,P)$ be a probability space with a right continuous and complete filtration
$\{\F_{s,t}\}_{0\, \leq s \leq\, t\,\leq T}$. Let $(X(t,x),  t \in [0,T], x \in U)$ be a continuous vectorfield valued semimartingale
with local characteristics $\tilde{A}:=(\tilde{a}_{ij}(t,x,y))_{i,j=1,\ldots,n}$ and $b:=(b_i(t,x))_{i=1\ldots,n}$  on $(\Omega,\F,P)$. This means that for every $x\in U$ the process $X(t,x)$ is a continuous $\R^n$ valued semimartingale
 with finite variation process $\int_0^t b(s,x)ds$ and quadratic variation 
\begin{gather}
  \langle X_i(t,x), X_j(t,y)\rangle \, = \, \int_0^t \tilde{a}_{ij}(s,x,y)ds \label{eq:quadrvar}. 
\end{gather}
 We assume that for every $(x,y)\in U\times U$ the function $\tilde{a}$ is continuous in time and of class $C^{4,\alpha}$ in both
space variables, and that  $x\mapsto b^i(t,x)$ is of class $C^{3,\alpha}$ for some $\alpha >0$.
Finally we assume that 
$\tilde{a}$ and $b$ have compact support in $U \times U$ resp. in $U$.

Denote by $(\varphi_{s,t}, s < t)$ the Stratonovich-Flow associated to $-X$.
 This means that almost surely $(\varphi_{s,t}, s < t)$ is a two parameter family of diffeomorphisms of $U$ fixing the boundary and verifying the flow property
\begin{gather}
  \varphi_{s,t} \circ   \varphi_{r,s}= \varphi_{r,t} \qquad \text{for $r \leq s \leq t$} \label{eq_Flow}.
\end{gather}
Furthermore, for every $x$ and every $s \in [0,T)$ the process $ (\varphi_{s,t}(x), t \geq s)$ is a solution of the 
stochastic differential equation
\begin{align}
d \varphi_{s,t}(x)\,= \,&  -X(\circ dt, \varphi_{s,t}(x)) \label{eq_Flowsde}\\
 \varphi_{s,s}(x)\,= \,& x \notag.
\end{align}
Under the above regularity assumption for all $s \leq t$ the mapping $\varphi_{s,t}$ is a $C^{3,\beta}$ diffeomorphism of 
$U$, for all $\beta < \alpha$. 

A particular example is that of a stochastic flow given by a usual Stratonovich-differential equation. If $X^{k}(t,\cdot), k= 0, \ldots, N$ 
are smooth time dependent vectorfields on $U$ and $B_1, \ldots B_N$ are independent standard Brownian motions, then 
\begin{gather}
 X(t,x)= \sum_{k=1}^N \int_0^t X^{k}(s,x) \circ d B_k(s) + \int_0^t X^{0}(s,x) ds, \label{eq_SDEBM}
\end{gather}
is a vectorfield valued Brownian motion as considered above. Its local characteristic is given by 
\begin{align}
 \tilde{a}_{ij}(s,x,y) \,=\,& \sum_{k=1}^N X^{k}_i(s,x) \, X^{k}_j(s,y)  \label{eq_loccharSDE}\\
b_i(s,x)\, =\,& X^{0}_i(s,x). \notag
\end{align}
In this case the stochastic differential equation \eqref{eq_Flowsde} reduces to 
the more familiar 
\begin{align}
d \varphi_{s,t}(x)\,= \,& X^{0}(t,\varphi_{s,t}(x)) dt + \sum_{k=1}^N X^{k}(t,\varphi_{s,t}(x)) \circ dB_k(t)  \label{eq_Flowsde2}\\
 \varphi_{s,s}(x)\,= \,& x \notag.
\end{align}
The advantage of Kunita's framework is that it allows for infinite sums in the noise part i.e. for noise fields of the form
\begin{gather}
 X(t,x)= \sum_{i=1}^\infty \int_0^t X^{k}(s,x)\circ d B_k(s) + \int_0^t X^{0}(s,x) ds, \label{eq_SDEBM2}
\end{gather}
for vectorfields with the right summability properties. We prefer this approach as a restriction to 
finite dimensional noises is unnecessary and a severe restriction.
\section{Results}
\label{sec:res}
In this section we state our main results. For the proofs see the
subsequent sections. We first address the question of existence and
uniqueness of solutions for \eqref{eq:sac}. There are some classical
existence and uniqueness results for equations similar to
\eqref{eq:sac}, see for example \cite[p. 212 ff.]{dPZ92},
\cite{KrRo07}, \cite{Fl96}. In
those references either mild or weak variational
solutions are constructed. 
Using the technique from \cite{Kun97} we obtain here
H\"older-continuous strong solutions.
\begin{theorem}\label{thm:existence}
Let $\ue^0$, $F$, and $X$ satisfy the assumptions \eqref{assumptions},
\eqref{eq:ass-init} and the smoothness conditions stated in Section
\ref{subsec:StochFlo}. Then for every $\eps>0$ there exists a unique
solution $u_\eps$ of 
\begin{align}
 u_\eps(t,x)\,&=\, u_\eps^0(x)+\int_0^t \Big(\Delta
 u_\eps(s,x)-\frac{1}{\eps^2}F'(u_\eps(s,x)) \Big) \,ds \,+\notag\\ 
 & \qquad \qquad \qquad +\int_0^t \nabla
 u_\eps(s,x) \cdot X(x,\circ ds) \label{eq:sac2}\\ 
  \nabla u_\eps\cdot \nu_U\,&=\, 0\quad \text{ on }(0,T)\times
  \partial U . \label{eq:sac-bdry}
\end{align}
The function $u_\eps(t,\cdot)$ is a continuous $C^{3,\beta}(\bar{U})$-valued 
semimartingale for any $0<\beta<\alpha$. Furthermore, we have the following bound for the spatial derivatives: 
\begin{align}
\Ex[\sup_{x \in U} |\partial^\gamma u_\eps(t,x)|^p]< \infty \label{eq:momboun},
\end{align}
for any multi-index $\gamma$ with $|\gamma| \leq 3$, every $p\geq1$, and for any $t\in (0,T)$. 
\end{theorem}
Our main result concerns the tightness of the solutions $u_\eps$ in the
limit $\eps \to 0$. In addition we show that limit points are concentrated on the space
of phase indicator function of bounded variation. The key step in the
proof of these results is a uniform bound on the diffuse surface area.
\begin{theorem}\label{thm:tightness}
Let the assumptions of Theorem \ref{thm:existence} be satisfied and let
$\ue$ be the solution of \eqref{eq:sac2}-\eqref{eq:sac-bdry} for $\eps>0$. Then the following
statements hold:
\setlength{\leftmargini}{4ex}
\renewcommand{\labelenumi}{(\arabic{enumi})}
\begin{enumerate}
\item  Uniform bounds on the energy: For every $T>0$ and every $p \geq 1$ we have
\begin{gather}
\sup_{\eps>0} \Ex \Big[ \sup_{0 \leq t \leq T} \En(\ue(t))^p \Big] < \infty. \label{eq:energy2_beta}
\end{gather}
\item Uniform bounds on the diffuse mean curvature: For every $T>0$ and every $p\geq 1$ we have
\begin{gather}
\sup_{\eps>0} \Ex \bigg[ \Big( \int_0^T \int_U \frac{1}{\eps} w_\eps (t,x)^2
\,dx dt \Big)^{p} \, \bigg] < \infty. \label{eq:energy_decay1_beta} 
\end{gather}
\item Tightness of the sequence: Let $\mathbb{Q}^\varepsilon$ be the
distribution of the solution \eqref{eq:sac}. Then the family
$\mathbb{Q}^\varepsilon$ is tight on $C([0,T], L^1(U))$. In particular,
there exists a sequence $\varepsilon_i \downarrow 0$ such that the
processes $u_{\varepsilon_i}$ can jointly be realized on a
probability space $(\tilde{\Omega},
\tilde{\mathcal{F}},\tilde{\mathbb{P}})$ and converge
$\tilde{\mathbb{P}}$-almost-surely in $C([0,T], L^1(U))$ to a
limiting process $u$. For almost all $t\in (0,T)$ we have $u(t,\cdot)\in BV(U,\{\pm 1 \}))$ almost surely and 
\begin{gather*}
  \Ex \bigg[  \sup_{0 \leq t \leq T} \|u(t)\|_{BV(U)}^p \bigg] < \infty \label{eq:energy_BV} 
\end{gather*}
holds for every $T>0$ and every $p\geq 1$.
\end{enumerate}
\end{theorem}
In most of the sequel we will use the It\^{o}-form
of \eqref{eq:sac2}, which is by \cite[Section 6.2]{Kun97} given as
\begin{align}
\ue(t,x)\,=\, \ue(0,x)+\int_0^t \Big(\Delta \ue(s,x)-\frac{1}{\eps^2}F'(\ue(s,x)) \Big)\, ds +\int_0^t \nabla \ue(s,x) \cdot X(ds, x) \notag\\
 + \frac{1}{2} \int_0^t \Big(A(s,x):D^2 \ue(s,x) + c(s,x)\cdot \nabla \ue(s,x)\Big)\,ds   \label{eq:sac3}.
 \end{align}
 Here and below we use the notation $A :B =\sum_{i,j} A_{ij} B_{ij}$ for the Hilbert-Schmidt scalar product of two matrices. 
 
The It\^{o}-Stratonovich correction terms in \eqref{eq:sac3} are given by the matrix $A=(a_{ij})_{i,j=1,...,n}$ and the vector field $c=(c_i)_{i=1,...,n}$, 
\begin{align}
  a_{ij}(t,x)\,&=\, \tilde{a}_{ij}(t,x,x), \label{eq:def-tilde-a}\\
  c_j(t,x)\,&=\,\partial_{y_i} \tilde{a}_{ij}(t,x,y)|_{y=x}, \label{eq:def-C}
\end{align}
where we sum here and in the following over repeated indices.
Note that the extra term
$A:D^2 u$ is of highest order, such  that it changes the diffusion
coefficient in \eqref{eq:sac2}. In particular, the
stochastic-parabolicity condition (see e.g.
\cite[condition (1)]{Fl96}), which is often needed in the case of
gradient dependend noise, is always satisfied.
\section{Existence and Uniqueness}
\label{sec:ExUn}
In this section we prove Theorem \ref{thm:existence} by reducing the
existence of solutions of \eqref{eq:sac} to  
an existence statement for a deterministic reaction-diffusion equation with random coefficients. This 
technique is borrowed from \cite{Kun97}. 
\begin{proof}[Proof of Theorem \ref{thm:existence}]
As above denote by $\varphi_{s,t}$ the stochastic flow generated by $-X$.
For a function $u:U \to \R$ define the  
transformation $w_t(u)(x)= u(\varphi_{0,t}^{-1}(x))$. By the regularity of
the stochastic flow it is clear that  
$w_t$ maps $C^{3,\beta}(U)$ into itself. Denote by $\mathcal{L}$ the nonlinear operator
\[
 \mathcal{L}(u)= \Delta u-\frac{1}{\varepsilon^2} F'(u)
\]
and by $\mathcal{L}^w_t$ the operator $w_t^{-1}  \mathcal{L} w_t$. Then
a direct computation shows that $\mathcal{L}^w_t$ is given by
\begin{gather}
 \mathcal{L}^w_t u(t,x)= \sum_{i,j=1}^n
 R_w^{i,j}(t,x)\frac{\partial^2}{\partial x^i  \partial x^j} u(t,x) +
 \sum_{i=1}^n S_w^i(t,x)\frac{\partial}{\partial x^i} u(t,x)- \frac{1}{\varepsilon^2} F'(u(t,x)),
 \end{gather}
with coefficients  
\begin{align}
R^{i,j}_w (t,x) \, &=\, \sum_{k,l} \partial_k  \Big( \varphi_{0,t}^{-1}\Big)^i(\varphi_{0,t}(x)) \partial_l  \Big( \varphi_{0,t}^{-1}\Big)^j(\varphi_{0,t}(x))\\
S^i_w(t,x)  \,&=\, \sum_{k,l}  \partial_k \partial_l \Big( \varphi_{0,t}^{-1} \Big)^i(\varphi_{0,t}(x))
\end{align}
In particular, the coefficients are random, the $R^{i,j}_w$ are of class $C^{3,\beta}$ 
in space and continuous $C^{\gamma}$ in time for every $\gamma < \frac{1}{2}$, 
and the $S^i$ are of class $C^{2,\beta}$ in space and $C^\gamma$ in time. Furthermore, 
note that $R^{i,j}=\delta^{i,j}$ and $S=0$ close to the boundary.
Similar to Lemma 6.2.3 in \cite{Kun97} it can be seen by another straightforward
computation that a smooth semimartingale $u$ is a solution  
to \eqref{eq:sac2} if and only if $u'=w_t^{-1} u$ is a solution to 
\begin{align}
\frac{\partial}{\partial t} v(t,x)= \mathcal{L}^w_t v(t,x) \label{eq:modsac}.
\end{align}
Existence and uniqueness of smooth solutions to reaction diffusion
equation like  
\eqref{eq:modsac} can be derived in a standard way: For example,
apply Schaefer's Fixed Point Theorem \cite[Theorem 9.2.4]{Evan98}
in combination with Schauder-estimates
\cite[Theorem IV.5.3]{LaSU68} for the linear part of \eqref{eq:modsac} and
a-priori estimates by the maximum principle. This yields existence of
solutions that are $C^{2,\beta}$ in space and $C^{1,\beta/2}$ in
time. Differentiation with respect to space and another application of
\cite[Theorem IV.5.3]{LaSU68} proves $C^{4,\beta}$-regularity in space.
To derive \eqref{eq:momboun} note that by \eqref{eq:modsac} the derivatives
of $u$ up to order 3 can be bounded in terms of the derivatives of $v$ and $\varphi^{-1}$.
The bounds on $v$ follow from the Schauder-estimates \cite[Theorem IV.5.3]{LaSU68} applied
to the random coefficients $R^{i,j}$ and $S^j$. The bounds on these coefficients as well
as on the derivatives of $\varphi$ follow from \cite[Theorem 6.1.10]{Kun97}.
\end{proof}
\section{Tightness}
\label{sec:Tightness}

In this section we derive estimates for the diffuse surface
area.

\begin{proposition}\label{prop:d-energy-calc}
Let $u_\eps$ satisfy \eqref{eq:sac}. Then for all $0\leq t_0<t_1$ and all $\eps>0$
\begin{align}
  & E_\eps(u_\eps(t_1)) - E_\eps(u_\eps(t_0))\notag\\
  = &-\int_{t_0}^{t_1}\int_U \frac{1}{\eps} w_\eps(t,x)^2 \,dx\,dt + \int_U \int_{t_0}^{t_1} w_\eps(t,x) \nabla u_\eps(t,x) \cdot X(dt,x) \,dx \notag\\
  &- \frac{1}{2} \int_{t_0}^{t_1} \int_U  w_\eps(t,x)\, c(t,x) \cdot \nabla u_\eps(t,x)  \,dx\,dt \notag\\
  &+ \frac12 
  \int_{t_0}^{t_1} \int_U {\eps} \,\nabla u_\eps(t,x)\cdot \Psi(t,x)\nabla\ue(t,x)+\psi(t,x)\frac{1}{\eps} F(\ue(t,x))dx dt\, \label{eq:d-energy}
\end{align}
holds with
\begin{align}
  \Psi_{ij}\,&=\, \Big( \partial_k \Big[  \partial_{x_l} \tilde{a}_{lk}(t;x,y)\big|_{x=y} -  \partial_{y_l} \tilde{a}_{lk}(t;x,y)\big|_{x=y} \Big] \, \,\Big) \delta_{ij}      \notag\\    
  &  +  \, \partial_{x_k} \partial_{y_k}  \tilde{a}_{ij}(t;x,y)\big|_{x=y}      + 2 \,\partial_j \big[ \partial_{x_k} \tilde{a}_{ik}(t;x,y)\big|_{x=y}  \big]                - 2 \,\partial_k \big[ \partial_{x_j} \tilde{a}_{ik}(t;x,y)\big|_{x=y} \big]  ,\notag \\
  \psi\,&=\,  \Big( \partial_j \Big[  \partial_{x_i} \tilde{a}_{ij}(t;x,y)\big|_{x=y} -  \partial_{y_i} \tilde{a}_{ij}(t;x,y)\big|_{x=y} \Big] \, \,\Big). \label{eq:def-Psi}
\end{align}
\end{proposition}
\begin{proof}
Recalling the conventions \eqref{eq:def-w_eps},  \eqref{eq:def-tilde-a}, and \eqref{eq:def-C} we can rewrite \eqref{eq:sac3} as 
\begin{align}
d\ue = -\frac{1}{\eps} w_\eps \, dt + \nabla \ue \cdot X(dt) + \frac12 A \colon D^2 \ue \, dt+ \frac12 c \cdot \nabla \ue \, dt \label{eq:new-Lemma1}.
\end{align}
It\^o's formula then yields 
\begin{align}
dE_\eps(\ue) =& \int_U w_\eps(t,x) \, d\ue(t,x) \, dx \notag\\
&+ \frac{1}{2} \Big[ \int_U \eps \tr(Q)(t,x,x) \, dx+ \int_U\frac{1}{\eps} F''\big(\ue(t,x)\big) \,  q(t,x,x) \, dx \Big] \,dt. \label{eq:new-Lemma2}
\end{align}
Here in the It\^o correction term we have used the short hand notation
\begin{align}
Q_{kl} (t,x,y)= & \partial_{x_k} \partial_{y_l}  \Big[ \partial_i \ue(t,x)\, \tilde{a}_{ij}(t;x,y) \, \partial_j \ue(t,y) \Big] \label{eq:new-Lemma3}\\
q(t,x,y) =& \partial_i \ue(t,x) \,\tilde{ a}_{ij}(t,x,y) \,  \partial_j \ue(t,y). \label{eq:new-Lemma4}
\end{align}
See \cite[Theorem 3.1.3]{Kun97} for a proof that the cross variation of $\nabla (\nabla \ue \cdot X )$ is indeed given by $Q$.  Recall that we use the convention to sum over repeated indices.

When evaluating the right-hand side of \eqref{eq:new-Lemma2} we obtain one `good' term, which corresponds to minus the integral over the squared diffuse mean curvature in the purely deterministic case. Additional terms are due to the stochastic drift in \eqref{eq:new-Lemma1} and the second order terms in the It\^o formula that can be found in the second line of \eqref{eq:new-Lemma2}. The objective is to show that these extra terms can finally be controlled by the `good' mean curvature term or by a Gronwall argument. Note that a priori this is not obvious at all, because on the right hand side of \eqref{eq:new-Lemma2} there appear terms like
\begin{align}
\int_U w_\eps(t,x) \,  A(t,x)\colon D^2 \ue(t,x) dx \, dt \label{eq:new-Lemma5}
\end{align}
that are not controlled by the mean curvature term.

We now start modifying the It\^o terms in the second line of \eqref{eq:new-Lemma2}. These calculations only involve the spatial integrals and we drop the time argument for simplicity. We start the calculation by treating the terms involving the potential $F$ on the right hand side of \eqref{eq:new-Lemma2}.
\begin{align}
\frac{1}{\eps}& \int_U F''\big(\ue(x) \big) \, \partial_i \ue(x)\,a_{ij}(x) \, \partial_j \ue(x) \, dx\notag\\
&= - \frac{1}{\eps} \int_U F'\big(\ue(x) \big) \,  \partial_i a_{ij}(x) \,\partial_j \ue(x) \, dx  - \frac{1}{\eps} \int_U F'\big(\ue(x) \big) \,   a_{ij}(x) \,\partial_i \partial_j \ue(x) \, dx\notag\\
&= - \frac{1}{\eps} \int_U F'\big(\ue(x) \big) \, \Big[  \partial_{x_i} \tilde{a}_{ij}(x,y)\big|_{x=y} -  \partial_{y_i} \tilde{a}_{ij}(x,y)\big|_{x=y} \Big] \,\partial_j \ue(x) \, dx \notag\\
& \qquad -  \frac{2}{\eps} \int_U F'\big(\ue(x) \big) \,    \partial_{y_i} \tilde{a}_{ij}(x,y)\big|_{x=y}  \,\partial_j \ue(x) \, dx  \notag\\
&\qquad - \frac{1}{\eps} \int_U F'\big(\ue(x) \big) \,   a_{ij}(x) \,\partial_i \partial_j \ue(x) \, dx\notag\\
&=  \frac{1}{\eps} \int_U F\big(\ue(x) \big) \, \partial_j \Big[  \partial_{x_i} \tilde{a}_{ij}(x,y)\big|_{x=y} -  \partial_{y_i} \tilde{a}_{ij}(x,y)\big|_{x=y} \Big] \, \, dx \notag\\
& \qquad -  \frac{2}{\eps} \int_U F'\big(\ue(x) \big) \,   \partial_{y_i} \tilde{a}_{ij}(x,y)\big|_{x=y}  \,\partial_j \ue(x) \, dx  \notag\\
&\qquad - \frac{1}{\eps} \int_U F'\big(\ue(x) \big) \,   a_{ij}(x) \,\partial_i \partial_j \ue(x) \, dx .\label{eq:new-Lemma6}
\end{align}
Here in the first step we have performed a partial integration with respect to the $x_i$ coordinate. In the second step we have added and subtracted an extra term. Finally, in the last line, we have performed another partial integration with respect to $x_j$ in the first term. 

In order to treat the terms involving $Q$ we start by evaluating the partial derivatives in  \eqref{eq:new-Lemma3}. We get
\begin{align}
Q_{kl} = T^1_{kl} + T^2_{kl} + T^3_{kl} +R_{kl}, \label{eq:new-Lemma7}
\end{align}
where 
\begin{align}
T^1_{kl}(x,y) = & \partial_k \partial_i \ue(x)  \, \tilde{a}_{ij}(x,y) \,  \partial_l \partial_j \ue(y),\notag\\
T^2_{kl} (x,y)= & \partial_i \ue(x) \, \partial_{x_k} \tilde{a}_{ij} (x,y) \partial_l \partial_j \ue(y)  ,\notag\\
T^3_{kl} (x,y)= &  \partial_k \partial_i \ue(x)\,  \partial_{y_l} \tilde{a}_{ij}(x,y)  \, \partial_j \ue(y) ,\notag\\
R_{kl} (x,y)= & \partial_i \ue(x) \,  \partial_{x_k} \, \partial_{y_l}  \tilde{a}_{ij}(x,y) \, \partial_j \ue(y) \label{eq:new-LemmaA}.
\end{align}
We have chosen to denote the last term with a different character $R$ to indicate that it will remain unchanged throughout the rest  of the calculation.

We start our calculation with the term in \eqref{eq:new-Lemma2} that involves $T^1$. By first performing first a partial integration in $x_j$ and then another partial integration in $x_k$ in the first term we get
\begin{align}
& \int_U \eps\,  \partial_k \partial_i \ue(x)  \, a_{ij}(x) \,  \partial_k \partial_j \ue(x) \, dx \notag\\
=& - \int_U \eps\,  \partial_j \partial_k \partial_i \ue(x)  \, a_{ij}(x) \,  \partial_k  \ue(x) \, dx - \int_U \eps\,  \partial_k \partial_i \ue(x)  \,\partial_j a_{ij}(x) \,  \partial_k  \ue(x) \, dx \notag\\
=&  \int_U \eps\,  \partial_j \partial_i \ue(x)  \, a_{ij}(x) \,  \Delta  \ue(x) \, dx  + \int_U \eps\,  \partial_j  \partial_i \ue(x)  \, \partial_k a_{ij}(x) \, \partial_k  \ue(x) \, dx \notag\\
& \qquad - \int_U \eps\,  \partial_k \partial_i \ue(x)  \,\partial_j a_{ij}(x) \,  \partial_k  \ue(x) \, dx .\label{eq:new-Lemma8}
\end{align}
The first term on the right hand side of \eqref{eq:new-Lemma8} has already the desired form. It cancels exactly with the corresponding part of the `bad' term mentioned above in \eqref{eq:new-Lemma5} . Let us treat the two extra terms on the right hand side of \eqref{eq:new-Lemma8} separately.  For the last term we get
\begin{align}
-& \int_U \eps\,  \partial_k \partial_i \ue(x)  \,\partial_j a_{ij}(x) \,  \partial_k  \ue(x) \, dx\notag\\
&= - \int_U \eps\,  \partial_k \partial_i \ue(x)  \,\Big[ \partial_{y_j} \tilde{a}_{ij}(x,y)\big|_{x=y}  - \partial_{x_j} \tilde{a}_{ij}(x,y)\big|_{x=y}  \Big] \,  \partial_k  \ue(x) \, dx \notag\\
&\qquad -  2\int_U \eps\,  \partial_k \partial_i \ue(x)  \,\partial_{x_j} \tilde{a}_{ij}(x,y)\big|_{x=y} \,  \partial_k  \ue(x) \, dx  \notag\\
&= \int_U \eps\,   \frac{1}{2} \big| \nabla  \ue(x) \big|^2  \partial_i \,\big[ \partial_{y_j} \tilde{a}_{ij}(x,y)\big|_{x=y}  - \partial_{x_j} \tilde{a}_{ij}(x,y)\big|_{x=y}  \big] \,  \, dx \notag\\
&\qquad +  2\int_U \eps\,   \partial_i \ue(x)  \,\partial_{x_j} \tilde{a}_{ij}(x,y)\big|_{x=y} \,  \Delta  \ue(x) \, dx\notag\\
&\qquad  +   2\int_U \eps\,  \partial_i \ue(x)  \, \partial_k \big[ \partial_{x_j} \tilde{a}_{ij}(x,y)\big|_{x=y}  \big]\,  \partial_k  \ue(x) \, dx . \label{eq:new-Lemma9}
\end{align}
Here in the second equality for the first term we have used the fact that $\partial_k \partial_i \ue \, \partial_k \ue = \frac12 \partial_i | \nabla \ue|^2$ to perform an integration by part. For the second term we have performed another integration by part in $x_k$.

For the second term on the right hand side of \eqref{eq:new-Lemma8} we add the terms involving $T^3$ and $T^4$ to obtain
\begin{align}
&\int_U \eps\,  \partial_j  \partial_i \ue(x)  \, \partial_k a_{ij}(x) \, \partial_k  \ue(x) \, dx + \int \eps \, \big( T^2_{k,k}(x,x) +  T^3_{k,k}(x,x) \big)  \, dx \notag\\  
&= \int_U \eps\,  \partial_j  \partial_i \ue(x)  \, \partial_{x_k} \tilde{a}_{ij}(x,y)\big|_{x=y} \, \partial_k  \ue(x) \, dx 
+ \int_U \eps  \, \partial_i \ue(x) \, \partial_{x_k} \tilde{a}_{ij} (x,y)\big|_{x=y} \partial_k \partial_j \ue(x) \, dx \notag\\
&+ \int_U \eps\,  \partial_j  \partial_i \ue(x)  \, \partial_{y_k} \tilde{a}_{ij}(x,y)\big|_{x=y} \, \partial_k  \ue(x) \, dx+  \int_U \partial_k \partial_i \ue(x)\,  \partial_{y_k} \tilde{a}_{ij}(x,y) \big|_{x=y} \, \partial_j \ue(x)  \, dx =\notag\\
&= -\int_U \eps\,   \partial_i \ue(x)  \, \partial_j \big[  \partial_{x_k} \tilde{a}_{ij}(x,y)\big|_{x=y} \big] \, \partial_k  \ue(x) \, dx\notag\\
& \qquad \qquad  -  \int_U \eps\,  \partial_j   \ue(x)  \,\partial_i \big[ \partial_{y_k} \tilde{a}_{ij}(x,y) \big|_{x=y} \big] \, \partial_k  \ue(x) \notag\\
&= - 2 \int_U \eps\,   \partial_i \ue(x)  \, \partial_j \big[ \partial_{x_k} \tilde{a}_{ij}(x,y)\big|_{x=y} \big] \, \partial_k  \ue(x) \, dx . \label{eq:new-Lemma10}
\end{align} 
Here for the second equality we have performed an integration by part in $x_j$ in the second line and another integration by part in $x_i$ in the third line. In the last equality we have used the identity $\tilde{a}_{ij}(x,y) = \tilde{a}_{ji}(y,x)$ which follows immediately from the definition \eqref{eq:quadrvar}.

So finally, summarising the calculations, we get by collecting and regrouping the terms from \eqref{eq:new-Lemma2}, \eqref{eq:new-Lemma6}, \eqref{eq:new-LemmaA}, \eqref{eq:new-Lemma8}, \eqref{eq:new-Lemma9}, and \eqref{eq:new-Lemma10} 
\begin{align}
dE_\eps(\ue(t)) &= \int_U w_\eps(t,x) \,  \Big[ -\frac{1}{\eps} w_\eps(t,x)  \, dt + \nabla \ue(t,x) \cdot X(dt) \notag\\ 
&\qquad \qquad + \frac12 A(t,x) \colon D^2 \ue(t,x) \, dt+ \frac12 c(t,x) \cdot \nabla \ue(t,x) \, dt    \Big] \,dx \notag\\
& - \frac{1}{2}\int_U  \frac{1}{\eps}  \big( F'\big(\ue(t,x) \big) -  \Delta  \ue(t,x) \big) \,   a_{ij}(t,x) \,\partial_i \partial_j \ue(t,x) \, dx \, dt\notag\\
&+  \frac{1}{2} \int_U \frac{1}{\eps} F\big(\ue(t,x) \big) \, \partial_j \Big[  \partial_{x_i} \tilde{a}_{ij}(t;x,y)\big|_{x=y} -  \partial_{y_i} \tilde{a}_{ij}(t;x,y)\big|_{x=y} \Big] \, \, dx\, dt \notag\\
&+  \frac12  \int_U \eps\,   \frac{1}{2} \big| \nabla  \ue(t,x) \big|^2  \partial_i \,\big[ \partial_{y_j} \tilde{a}_{ij}(t;x,y)\big|_{x=y}  - \partial_{x_j} \tilde{a}_{ij}(t;x,y)\big|_{x=y}  \big] \,  \, dx  \, dt\notag\\
&  -  \frac{1}{\eps} \int_U F'\big(\ue(t,x) \big) \,    \partial_{y_i} \tilde{a}_{ij}(t;x,y)\big|_{x=y}  \,\partial_j \ue(t,x) \, dx \, dt  \notag\\
& +  \int_U \eps\,   \partial_i \ue(t,x)  \, \partial_{x_j} \tilde{a}_{ij}(t;x,y)\big|_{x=y} \,  \Delta  \ue(t,x) \, dx\notag\\
& + \frac12 \int_U \eps \,  \partial_i \ue(t,x) \,  \partial_{x_k} \partial_{y_k}  \tilde{a}_{ij}(t;x,y)\big|_{x=y} \, \partial_j \ue(t,x) \, dx \, dt\notag\\
&  +   \int_U \eps\,  \partial_i \ue(t,x)  \, \partial_k \big[ \partial_{x_j} \tilde{a}_{ij}(t;x,y)\big|_{x=y}  \big]\,  \partial_k  \ue(t,x) \, dx \, dt \notag\\
&-  \int_U \eps\,   \partial_i \ue(t,x)  \, \partial_j \big[ \partial_{x_k} \tilde{a}_{ij}(t;x,y)\big|_{x=y} \big] \, \partial_k  \ue(t,x) \, dx \, dt .\label{eq:new-Lemma11}
\end{align}
As mentioned above  the terms in the third line of \eqref{eq:new-Lemma11} cancel with the first term in the second line. Also note that using the symmetry $\tilde{a}_{ij}(x,y) = \tilde{a}_{ji}(y,x)$ once more we can see that the terms in square brackets in the fourth and fifth lines are identical, so that these two terms can be written as 
\begin{align}
\mu_\eps^t \Big( \partial_j \Big[  \partial_{x_i} \tilde{a}_{ij}(t;x,y)\big|_{x=y} -  \partial_{y_i} \tilde{a}_{ij}(t;x,y)\big|_{x=y} \Big] \, \,\Big) \label{eq:new-Lemma12}.
\end{align}
Finally, using the symmetry of $\tilde{a}$ once more and recalling the definition \eqref{eq:def-C}, we see that $\partial_{x_j}\tilde{a}_{ij} = \partial_{y_i}\tilde{a}_{ij} = c_j$.  Hence, the terms in the sixth and seventh lines yield $-2$ times the last term in the second line and the sign in  this It\^o-Stratonovich correction term changes. 
Putting everything together we get the identity
\begin{align}
  E_\eps(u_\eps(t_1))  -& E_\eps(u_\eps(t_0))   \notag\\
  &= - \int_{t_0}^{t_1} \int_U\frac{1}{\eps} w_\eps(t,x)^2  \,dx  \, dt +  \int_U\int_{t_0}^{t_1}  w_\eps(t,x) \, \nabla \ue(t,x) \cdot X(dt,x) \,  dx\notag\\ 
&- \frac12  \int_{t_0}^{t_1} \int_U c(t,x) \cdot \nabla \ue(t,x) w_\eps(t,x)\, dx \, dt   \notag\\
&+ \frac 12 \int_{t_0}^{t_1} \mu_\eps^t \Big( \partial_j \Big[  \partial_{x_i} \tilde{a}_{ij}(t;x,y)\big|_{x=y} -  \partial_{y_i} \tilde{a}_{ij}(t;x,y)\big|_{x=y} \Big] \, \,\Big) \, dt\notag\\
& + \frac12 \int_{t_0}^{t_1} \int_U  \eps \,  \partial_i \ue(t,x) \,  \partial_{x_k} \partial_{y_k}  \tilde{a}_{ij}(t;x,y)\big|_{x=y} \, \partial_j \ue(t,x) \, dx \, dt\notag\\
 &  +   \int_U \eps\,  \partial_i \ue(t,x)  \, \partial_k \big[ \partial_{x_j} \tilde{a}_{ij}(t;x,y)\big|_{x=y}  \big]\,  \partial_k  \ue(t,x) \, dx \, dt \notag\\
&-  \int_U \eps\,   \partial_i \ue(t,x)  \, \partial_j \big[ \partial_{x_k} \tilde{a}_{ij}(t;x,y)\big|_{x=y} \big] \, \partial_k  \ue(t,x) \, dx \, dt .\label{eq:d-energy1}
\end{align}
This equality is equivalent to \eqref{eq:d-energy}.
\end{proof}
\begin{remark} 
In the special case that the Brownian vector field $X(\cdot,\circ dt)$ is given by \eqref{eq_SDEBM}, the equation \eqref{eq:d-energy1} yields
\begin{align}
  &E_\eps(u_\eps(t_1)) - E_\eps(u_\eps(t_0))   \notag\\
  =\,& - \int_{t_0}^{t_1} \int_U\frac{1}{\eps} w_\eps^2  \,dx  \, dt +  \int_U\int_{t_0}^{t_1}  w_\eps \, \nabla \ue \cdot X^k\,dB_k(t) \,  dx\notag\\
  &+  \int_{t_0}^{t_1} \int_U  w_\eps \, \nabla \ue \cdot X^0\,dx \,  dt\notag\\ 
\,&- \frac12  \int_{t_0}^{t_1} \int_U DX^kX^k \cdot \nabla \ue w_\eps\, dx \, dt   \notag\\
\,&+ \frac 12 \int_{t_0}^{t_1} \mu_\eps^t \Big( (\nabla\cdot X^k)^2 - \tr (DX^k DX^k) \,\Big) \, dt\notag\\
\,& + \int_{t_0}^{t_1} \int_U   \frac{\eps}{2}  |DX^T\nabla\ue|^2 + \eps\nabla\ue\cdot DX^kDX^k\nabla\ue \, dx \, dt\notag\\
\,& -  \int_{t_0}^{t_1} \int_U \eps\,  (\nabla\ue\cdot DX^k\nabla\ue)\nabla\cdot X^k \, dx \, dt. \label{eq:energy-spca}
\end{align}
We next remark that the first \emph{inner} variation of $E_\eps$ at $u_\eps$ in direction of a vector field $Y$ is given by
\begin{gather*}
	\delta E_\eps(u_\eps)(Y) \,=\, \int_U -w_\eps \nabla \ue\cdot Y\,dx,
\end{gather*}
and that we obtain from \cite{Le10} for the second {inner} variation of $E_\eps$ at $u_\eps$ in direction of a vector field $Y$
\begin{align*}
	\delta^2 E_\eps(u_\eps)(Y,Y) \,=\, & \Big(\frac{\eps}{2}|\nabla\ue|^2 + \frac{1}{\eps}F'(\ue)\Big) \Big( (\nabla\cdot Y)^2 - \tr (DY DY) \Big) +\\*
	 &+   \frac{\eps}{2}  |DX^T\nabla\ue|^2 + \eps\nabla\ue\cdot DYDY\nabla\ue -    (\nabla\ue\cdot DY\nabla\ue)\nabla \cdot Y. 
\end{align*}
Therefore, we can rewrite \eqref{eq:energy-spca} as 
\begin{align*}
  E_\eps(u_\eps(t_1)) - E_\eps(u_\eps(t_0))   
  =\,& - \int_{t_0}^{t_1} \int_U\frac{1}{\eps} w_\eps^2  \,dx  \, dt\notag\\
  &+  \int_{t_0}^{t_1}\delta E_\eps(u_\eps(t))\big(-X^k \big) \circ dB_k(t)  \notag\\
  &+  \int_{t_0}^{t_1}\delta E_\eps(u_\eps(t)) (-X^0)\, dt  \\
  & +\frac 12 \int_{t_0}^{t_1}  \delta^2 E_\eps(u_\eps(t))(X^k,X^k)\,dt,
\end{align*}
which reflects the fact that the perturbation $\nabla\ue\cdot X^k 
\circ dB_k$ to the Allen-Cahn equation in \eqref{eq:sac} corresponds to an inner variation of $\ue$ in direction of  $-X^k \circ dB_k(t)$.
\end{remark}
With Proposition \ref{prop:d-energy-calc} in hand we are now ready to derive the desired moment estimates \eqref{eq:energy2_beta} and \eqref{eq:energy_decay1_beta}. We give the proof only for $p \geq 2$. The general case then follows from Young's inequality.
\begin{proposition}
For every $p \geq 2$ and all $T>0$ we have
\begin{align}
\sup_{\eps > 0} \Ex \Big[ \sup_{0 \leq t \leq T} \En(\ue(t))^p \Big] &< \infty \label{eq:energy2} \qquad \text{and} \\
\sup_{\eps > 0} \Ex \bigg[ \Big( \int_0^T \int_U \frac{1}{\eps} w_\eps ^2(t,x) \,dx dt \Big)^{p} \, \bigg] &< \infty. \label{eq:energy_decay1}
\end{align}
\end{proposition}
\begin{proof}
We start by observing that due to the positivity of both summands we can write 
\begin{align} 
 \Ex \Big[ \sup_{0\leq s \leq t } &E_\eps(u_\eps(s))^p \Big] + \Ex \bigg[\bigg( \int_0^t  \int_U \frac{1}{\eps} w_\eps ^2(s,x) \,dx ds  \bigg)^p \bigg] \notag\\
 &\leq 2^{p-1} \, \Ex \bigg[ \bigg(\sup_{0\leq s \leq t } E_\eps(u_\eps(s)) +  \int_0^t  \int_U \frac{1}{\eps} w_\eps ^2(s,x) \,dx ds  \bigg)^p \bigg] .
\label{eq:new8}
\end{align}
According to \eqref{eq:d-energy} we can write for any $t >0$
\begin{equation}\label{eq:new1}
 E_\eps(u_\eps(t)) + \int_0^t  \int_U \frac{1}{\eps} w_\eps ^2(t,x) \,dx dt = \En\big( u_\eps^0 \big) + M_\eps(t) +H_\eps(t),
\end{equation}
where
\begin{align}\label{eq:new2}
M_\eps(t) =  \int_U \int_{0}^{t} w_\eps(t,x) \nabla u_\eps(t,x) \cdot \tilde{X}(dt,x) \,dx ,
\end{align}
and 
\begin{align}
H_\eps(t) =&  \int_{0}^{t} \int_U  w_\eps(s,x)\, \tilde{b}(s,x) \cdot \nabla u_\eps(s,x)  \,dx\,ds  \label{eq:new3}\\
  &+ \frac12 
  \int_{0}^{t_1} \int_U {\eps} \,\nabla u_\eps(s,x)\cdot \Psi(s,x)\nabla\ue(s,x)+\psi(t,x)\frac{1}{\eps} F(\ue(s,x)) \,dx \, ds \,.\notag
\end{align}
Here we have used the notation
\begin{equation}\label{eq:tildeX}
\tilde{X}(t,x) = X(t,x) -\int_0^t b(s,t) \, ds\quad \text{and}   \quad \tilde{b}(t,x) = b(t,x) - \frac12 c(t,x).
\end{equation}
Then $M_\eps$ is a local martingale and $H_\eps$ is a process of bounded variation. The quadratic variation of $M_\eps$ is given by
\begin{gather}
  \dft \langle M_\eps \rangle_t \, = \, \int_U \int_U w_\eps(t,x) \nabla \ue(t,x) \tilde{A}(t,x,y) w_\eps(t,y) \nabla \ue(t,y) dx \, dy \, dt  .
\end{gather}

Then recalling \eqref{eq:ass-init} we can write for any $\eps>0$ and any $p\geq 1$ and any $t>0$ that 
\begin{align}
\Ex \bigg[ \bigg(\sup_{0\leq s \leq t } &E_\eps(u_\eps(s)) + \int_0^t  \int_U \frac{1}{\eps} w_\eps ^2(s,x) \,dx ds  \bigg)^p \bigg] \notag\\
& \leq 3^{p-1} \Lambda^p + 3^{p-1} \, \Ex \Big[ \sup_{0 \leq s \leq t} \big| M_\eps(s) \big|^p\Big] + 3^{p-1} \,\Ex \Big[ \sup_{0 \leq s \leq t} \big| H_\eps(s) \big|^p\Big].  \label{eq:new4}
\end{align}
For the martingale term we get using the Burkholder-Davis-Gundy inequality \cite[Theorem 3.28  on page 166]{KS91} that
\begin{align}\label{eq:new5}
\Ex \Big[ \sup_{0 \leq s \leq t} \big| M_\eps(s) \big|^p\Big] \leq C \, \Ex \Big[  \langle M_\eps \rangle_t^{\frac{p}{2}} \Big],
\end{align}
for a constant $C$ depending only on $p$. In order to bound this expression we use
\begin{align}
w_\eps \nabla u_\eps \,= \,& - \eps \Delta u_\eps \nabla u_\eps + \frac{1}{\eps} F'(u_\eps) \nabla u_\eps  =\notag\\
=\,& - \nabla \cdot \big(\eps \nabla u_\eps \otimes \nabla u_\eps  \big) + \nabla \Big( \frac{\eps}{2} \big| \nabla u_\eps \big|^2 + \frac{1}{\eps} F(u_\eps) \Big) \label{eq:prop_id14};
\end{align}
a partial integration yields
\begin{align}
&\int_U \tilde{A}(s,x,y) w_\eps(s,y) \nabla \ue(s,y) \, dy \,\notag\\
=\,& \int_U \nabla_y \tilde{A}(s,x,y) \big(  \eps \nabla \ue(s,y) \otimes \nabla \ue(s,y) \big)\, dy \notag\\
 &- \int_U \nabla_y \tilde{A}(s,x,y) \Big( \frac{\eps}{2} \big| \nabla \ue(s,y) \big|^2 + \frac{1}{\eps} F(\ue(s,y)) \Big)\,dy. \label{eq:prop_id15}
\end{align}
Thus repeating the same partial integration in the $x$-variable we can conclude that
\begin{align}
  \Big|\int_U & \int_U w_\eps(s,x) \nabla \ue(s,x) \cdot \tilde{A}(s,x,y) w_\eps(s,y) \nabla \ue(s,y) dx \, dy  \bigg| \notag\\
  \leq\,& 4 \|\tilde{A} \|_{C^0([0,T];C^2(\bar{U}\times \bar{U}))} \En(\ue(s))^2. \label{eq:prop_id16}
\end{align}
Hence for any $t \leq T$  the right hand side of  \eqref{eq:new5} can be bounded by
\begin{align}
 \Ex \Big[  \langle M_\eps \rangle_t^{\frac{p}{2}} \Big]    \leq& 4 \|\tilde{A} \|_{C^0([0,T];C^2(\bar{U}\times \bar{U}))} \, \Ex \bigg[  \Big( \int_0^t \En(\ue(s))^2  ds  \Big)^{\frac{p}{2}} \bigg] \notag\\
 \leq & 4  \|\tilde{A} \|_{C^0([0,T];C^2(\bar{U}\times \bar{U}))} T^{\frac{p-2}{2}} \,  \int_0^t \Ex \big[  \En(\ue(s))^p  \big]  \,ds. 
 \end{align}
Here in the second line we have used  H\"older inequality and Fubini theorem. For the term $H_\eps$ we write
\begin{align}
\Ex \Big[ & \sup_{0 \leq s \leq t} | H_\eps(s)|^p \Big] \label{eq:new6}\\
\leq&  2^{p-1}  \,\Ex\Big[ \Big| \int_0^t \int_U  w_\eps(r,x) \nabla u_\eps(r,x)\cdot \tilde{b}(r,x)  \,dx \, dr \Big|^p  \Big] \notag\\
  &+ 2^{p-1} \,
  \Ex\Big[ \Big| \int_0^t  \int_U \frac{\eps}{2} \nabla u_\eps(r,x)\cdot \Psi(r,x)\nabla\ue(r,x)+\psi(r,x)\frac{1}{2\eps} F(\ue(r,x))\,dx \, dr \Big|^p \Big]. \notag
\end{align}
Using Young's inequality for any $\delta>0$ the first term on the right hand side of \eqref{eq:new6} can be estimated by 
\begin{align}
  \Ex & \bigg[\Big|\int_0^t \int_U  w_\eps \nabla
	u_\eps \cdot \tilde{b}\,dx \, ds \Big|^{p} \bigg] \notag\\
	&\leq\, \| \tilde{b} \|_{C^0([0,T] \times U)}^p  \,  \Ex \bigg[\Big|\int_0^t \int_U \frac{\delta}{2\eps} w_\eps^2 + \frac{\eps}{2\delta} |\nabla \ue|^2 \, dx \, ds \Big|^{p} \bigg] \notag\\
	&\leq\, \delta^p \frac12 \| \tilde{b} \|_{C^0([0,T] \times U)}^p  \Ex \bigg[ \Big|\int_0^t \int_U \frac{1}{\eps} w_\eps^2 \, dx \, ds \Big|^{p} \bigg]  \notag\\
	&\qquad + \frac{1}{\delta^p}   2^{p-1} \| \tilde{b} \|_{C^0([0,T] \times U)}^p T^{{p-1}} \int_0^t \Ex \bigg[ \Big(\int_U \frac{\eps}{2} |\nabla \ue|^2 \, dx \, \Big)^{p} \bigg] ds. \label{eq:energy_decay4.1}
\end{align}
For the second term on the right hand side of \eqref{eq:new6} we write 
\begin{align}
\Ex\Big[ \Big| \int_s^t  \int_U \frac{\eps}{2} \nabla u_\eps(r,x)\cdot \Psi(r,x)\nabla\ue(r,x)+\psi(r,x)\frac{1}{2\eps} F(\ue(r,x))\,dx \, dr \Big|^p \Big] \notag\\
\leq  \big( \| \psi \|_{C^0([0,T] \times U)} + \| \Psi \|_{C^0([0,T] \times U)}\big)^p T^{{p-1}} \int_0^t \Ex \big[\En(\ue(s))^p \, \big] ds.
\label{eq:new7}
\end{align}
So finally,  \eqref{eq:new8} and \eqref{eq:new4}--\eqref{eq:new7} we obtain 
\begin{align}
 \Ex \Big[ \sup_{0\leq s \leq t } &E_\eps(u_\eps(s))^p \Big] + \Ex \bigg[\bigg( \int_0^t  \int_U \frac{1}{\eps} w_\eps ^2(s,x) \,dx ds  \bigg)^p \bigg]  \label{eq:new9}\\
 & \leq C_1 \Big(1 +\frac{1}{\delta^p} \Big) \int_0^t \Ex \Big[  E_\eps(u_\eps(s))^p \Big] \, ds+ C_2 \delta^p \Ex \bigg[\bigg( \int_0^t  \int_U \frac{1}{\eps} w_\eps ^2(s,x) \,dx ds  \bigg)^p \bigg] \notag,
\end{align}
for positive constants $C_1 =C_1(p,T,\|b\|_{C^0([0,T]\times U)},\|\tilde{A}\|_{C^0([0,T];C^2(U))})$ and $C_2=C_2(p,T,  \| \tilde{b} \|_{C^0([0,T] \times U)})$. Hence choosing $\delta =\Big( \frac{1}{C_2}\Big)^{\frac{1}{p}}$  and applying Gronwall's Lemma we obtain \eqref{eq:energy2}. Then applying the same estimate for $\delta =\frac12 \Big( \frac{1}{C_2}\Big)^{\frac{1}{p}}$ and using \eqref{eq:energy2} we also obtain \eqref{eq:energy_decay1}.
\end{proof}

As in \cite{Modi87} it is convenient to consider instead of $u_\eps$ the transformation $G(u_\eps)$ given by the function $G(s):=\int_0^s \sqrt{2 F(r)} dr$. Note that $G$ is smooth, increasing and that $G(0)=0$. We will need the following bound on the increments of $ G(\ue(t))$:
\begin{lemma}
For every smooth test function $\varphi \in C_b^\infty(U)$ and every $p \in \N$ we have for all $0 \leq s < t \leq T$
\begin{gather}
\Ex \Big[ \Big| \int_U G(\ue(t,x)) \varphi(x) \, dx - \int_U G(\ue(s,x)) \varphi(x) \, dx  \Big|^{2p} \Big] \leq C |t-s|^{p}, \label{eq:increments1}
\end{gather}
where $C=C(p,\|\tilde{A}\|_{C^0([0,T]\times U\times U)},\|A\|_{C^0([0,T];C^1(U))},\|b,c\|_{C^0([0,T]\times U)},\|\varphi\|_{C^1(U)})$.
\end{lemma}
\begin{proof}
Noting that  $G'(\ue)=\sqrt{2F(\ue)}$,  It\^{o}'s formula and \eqref{eq:sac3} imply that 
\begin{align}
\int_U & G(\ue(t,x))  \varphi(x) \, dx - \int_U G(\ue(s,x)) \varphi(x) \, dx  \notag\\
\,= \,& \int_U \Big[ \int_s^t  \varphi(x) \sqrt{2F(\ue(r,x))} \bigg[\Big(  \Delta \ue(r,x)-\frac{1}{\eps^2}F'(\ue(r,x)) \Big) \,dr\notag\\
&+\int_s^t \nabla \ue(r,x) \cdot X(dr,x) \notag\\
 &+ \int_s^t\frac{1}{2}  A(r,x):D^2 \ue(r,x)\,dr + \frac{1}{2}c(r,x)\cdot \nabla \ue(r,x) \,dr    \bigg] \, dx \label{eq:increments2} \\
&+\frac{1}{2} \int_s^t \int_U  G''(\ue(r,x))  \varphi(x) \nabla \ue(r,x)  \cdot A(r,x)  \nabla \ue(r,x) dx \, dr \notag.
\end{align}
Thus for $p \in \N$ one can write
\begin{gather}
\Ex \Big[ \Big| \int_U G(\ue(t))\, \varphi \, dx - \int G(\ue(s)) \, \varphi \, dx  \Big|^{2p} \Big] \leq 4^{p-1}(I_1 +I_2 + I_3 + I_4 ). \label{eq:increments3}
\end{gather}
Let us bound the individual terms:
\begin{align}
I_1 \,= \, & \Ex \bigg[\Big| \int_s^t \int_U \varphi(x) \, \sqrt{2F(\ue(r,x))} \, \Big[ \Delta \ue(r,x) -\frac{1}{\eps^2}F'(\ue(r,x)) \Big] \,dx \, dr \Big|^{2p} \bigg] \,\leq \notag\\
\leq \, & \| \varphi \|_{C^0(U)}^{2p} \, \Ex \bigg[ \Big( \int_s^t \int_U \frac{1}{\varepsilon} \, F(\ue(r,x)) \,dx dr \Big)^{p} \Big( \int_s^t \int_U \frac{1}{\eps} w_\eps(r,x)^2 \,dx dr \Big)^{p}  \bigg] \notag\\
\leq \, & \| \varphi \|_{C^0(U)}^{2p} \, \Ex \bigg[ \Big(\int_s^t \En(\ue(r,\cdot) dr \Big)^{2p} \bigg]^{1/2} \, \Ex \bigg[ \Big( \int_s^t \int_U \frac{1}{\eps} w_\eps(r,x)^2 \,dx dr \Big)^{2p} \, \bigg]^{1/2}  \notag\\
\leq \, & \| \varphi \|_{C^0( U)}^{2p} \, (t-s)^p \, \Ex \bigg[ \Big(\sup_{s \leq r \leq t} \En(\ue(r,\cdot) \Big)^{2p} \bigg]^{1/2} \Ex \bigg[ \Big( \int_s^t \int_U \frac{1}{\eps} w_\eps(r,x)^2 \,dx dr \Big)^{2p} \, \bigg]^{1/2} \notag\\
\leq \,& C (t-s)^p. \label{eq:increments4}
\end{align}
Here in the second and third line we have used Cauchy-Schwarz inequality. In the last line we have used \eqref{eq:energy2} and \eqref{eq:energy_decay1}. The second term can be bounded using Youngs inequality:
\begin{align}
I_2 \, = \, & \Ex \bigg[ \Big| \int_s^t \int_U \varphi \, \sqrt{2F(\ue(r,x))} \, \Big[ \nabla \ue(r,x) \,\cdot \,\big(\frac{1}{2}c(r,x)+b(r,x)\big) \Big] \,dx \, dr \Big|^{2p} \bigg] \notag \\
\leq \, & \| \varphi \|_{C^0( U)}^{2p} \,\Big\| \frac12 c+b \Big\|_{C^0([0,T] \times U)}^{2p} \, \Ex \bigg[ \Big| \int_s^t \En(\ue(r,\cdot)) dr \Big|^{2p} \bigg] \label{eq:increments5} \\
\leq \, & C (t-s)^{2p} \notag.
\end{align}
Here we have used the inequalities $|\nabla u_\eps \sqrt{2F(u_\eps)}|\leq \frac{\eps}{2}|\nabla u_\eps|^2+\frac{1}{\eps}F(u_\eps)$ and \eqref{eq:energy2} in the last line. For the martingale term we get using Burkholder-Davis-Gundy inequality in the second line and then Youngs inequality in the third line:
\begin{align}
I_3 \, = \, & \Ex \bigg[ \Big| \int_s^t \int_U \varphi(x) \, \sqrt{2F(\ue(r,x))} \, \Big[  \nabla \ue(r,x) \cdot \big( X(dr,x)-b(r,x)dr \big) \, \Big] \,dx \Big|^{2p} \bigg] \notag \\
\leq \, & \Ex \bigg[  \int_s^t \, \Big| \, \int_U \int_U \varphi(x) \, \sqrt{2F(\ue(r,x))} \nabla \ue(r,x) \cdot \notag\\
&\qquad \cdot\tilde{A}(r,x,y) \, \nabla \ue(r,y) \sqrt{2F(\ue(r,y))} \varphi(y)\,dx \,dy \, dr \Big|^{p} \bigg] \notag\\
\leq \,& \| \varphi \|_{C^0( U)}^{2p} \|\tilde{A} \|_{C^0([0,T] \times U)}^{p} \Ex \bigg[\Big( \int_s^t \, \Big( \,  \int_U  \, \sqrt{2F(\ue(r,x))} |\nabla \ue(r,x)|\,dx \Big)^2\,  dr \Big)^{p} \bigg] \notag\\ 
\leq \,& \| \varphi \|_{C^0( U)}^{2p} \|\tilde{A} \|_{C^0([0,T] \times U)}^{p}  \, \Ex \bigg[ \Big| \int_s^t \En(\ue(r,\cdot))^2 dr \Big|^{p} \bigg] \notag \\
\leq \, & C(t-s)^{p}. \label{eq:increments6} 
\end{align}
For the fourth term we write:
\begin{align}
I_4 \, = \, & \frac{1}{2}\, \Ex \bigg[ \Big|  \int_s^t \int_U \varphi(x) \, \sqrt{2F(\ue (r,x))} \, \Big( A(r,x): D^2 \ue(r,x) \Big) \notag\\
 &+ \, \varphi(x) G''(\ue(r,x)) \nabla \ue(r,x) \cdot A(r,x) \nabla \ue(r,x)  \, dx \, dr \Big|^{2p} \bigg] \label{eq:increments7} .
\end{align}
After a partial integration the second summand can be written as:
\begin{align}
 \int_U & \varphi(x) G''(\ue(r,x) \nabla \ue(x) \cdot A(r,x) \nabla \ue(x)  \, dx  \,=\notag\\
 =\, & \int_U \varphi(x) \, \nabla \Big(\sqrt{2F(\ue(r,x))} \Big)\cdot  A(r,x) \nabla \ue(x) dx \notag \\
=\, & - \int_U \nabla \varphi(x)   \cdot  A(r,x) \nabla \ue(x)  \sqrt{2F(\ue(r,x))} \,dx  \notag\\
&- \int_U \varphi(x) \, \sqrt{2F(\ue(r,x))} \big(\nabla \cdot \, A(r,x) \big)  \nabla \ue(r,x)  \, dx\notag\\
& -  \int_U \varphi(x) \, \sqrt{2F(\ue(r,x))} \, A(r,x) : D^2 \ue(r,x)  \, \Big)\, dx. \label{eq:increments8}
\end{align}
Noting that the terms involving $D^2 u$ in \eqref{eq:increments7} and \eqref{eq:increments8} cancel it remains to bound:
\begin{align}
\Ex \bigg[ & \Big| \int_s^t \int_U \nabla \varphi(x)   \cdot  A(r,x) \nabla \ue(x)  \sqrt{2F(\ue(r,x))} \,dx \notag\\
& + \int_s^t \int_U \varphi(x) \, \sqrt{2F(\ue(r,x))} \big(\nabla \cdot \, A(r,x) \big)  \nabla \ue(r,x)  \, dx \, ds \Big|^{2p} \bigg] \notag \\
&\leq \, \Big( 2 \|A \|_{C^0([0,T] ,C^1(U)} \| \varphi \|_{C^1(U))} \Big)^{2p} \Ex \bigg[ \Big| \int_s^t \En(\ue(r)) dr \Big|^{2p} \bigg] \notag\\
& \leq \, C (t-s)^{2p} \label{eq:increments9}.
\end{align}
This finishes the proof.
\end{proof}

Now we are ready to prove  our main theorem.
\begin{proof}[Proof of Theorem \ref{thm:tightness}]
As a first step we will show that the distributions of $G(\ue)$ are tight on $C([0,T], L^1(U))$. To this end it suffices to show the following two conditions (see e.g.  \cite[Theorem 3.6.4, page 54 ]{Da93} together with \cite[Theorem 8.3, page 56]{Bi99}. Note that conditions (8.3) and (8.4) in \cite{Bi99} are implied by our stronger assumption (ii)):
\begin{enumerate}
\item[(i)] (\emph{Compact Containment}) For every $\delta > 0$ there exists a compact set $K^\delta \subseteq L^1(U)$ such that for all $\eps > 0$
\begin{gather}
\Prob \Big[G(\ue(t)) \in K^\delta \, \text{for all} \, 0 \leq t \leq T \Big] \geq 1 - \delta
\end{gather}
\item[(ii)] (\emph{Weak tightness}) For every smooth testfunction $\varphi \in C^\infty(U)$ there exist positive $\alpha, \beta, C$ such that for all $0 \leq s < t \leq T$
\begin{gather}
\Ex \Big[ \Big| \int_U G(\ue(t,x)) \varphi(x) \, dx  - \int G(\ue(s,x)) \varphi(x) \,dx \Big|^{\alpha} \Big] \leq C |t-s|^{1+\beta}.
\end{gather}
\end{enumerate}
To prove the first statement note that
\begin{align}
\int_U \big| \nabla G(\ue(t,x)) \big| \, d x \, = \, & \int_U \big| \sqrt{2 F(\ue(t,x))} \big| \nabla \ue(t,x) \big| \, d x \notag \\
\leq\,& \int_U \frac{\eps}{2} \big| \nabla \ue(t,x) \big|^2 + \frac{1}{\eps} F(\ue(t,x)) \, dx.
\end{align}
Furthermore, \eqref{assumptions} implies $G(r) \leq C ( 1 + F(r))$ such that
\begin{align}
\int_U \big| G(\ue(t,x)) \big| \, dx \, \leq \, & C \int_U F(\ue(t,x)) dx + C |U|.
\end{align}
Thus applying Chebyshev's inequality equation \eqref{eq:energy2} yields  that
\begin{gather}
\Prob \Big[\sup_{ 0 \leq t \leq T} \| G(\ue(t) \|_{W^{1,1}(U)} \geq \lambda \Big] \to 0,
\end{gather}
for $\lambda \to \infty$ which together with Rellich's Theorem implies condition (i).

The second condition (ii) follows from \eqref{eq:increments1}. The tightness of $G(\ue)$ is thus proved. 

\vskip3ex

As a second step we prove that the distributions of $u_\eps$ are tight in $C([0,T],L^1(U))$. Denote by $G^{-1}$ the inverse function of $G$. Then similar to \cite[page 139]{Mo87} we observe that the operator $v \mapsto  G^{-1} \circ v$ is continuous from $L^1(U)$ to itself. In fact assume $v_i \to v$ as $i\to\infty$ in $L^1(U)$. The growth condition in \eqref{assumptions} implies that $G^{-1}$ is uniformly continuous on $\R$. This implies convergence in measure and convergence pointwise in $U$ for the sequence $G^{-1} \circ v_i$. Furthermore, using the growth condition once more one can see that 
$|G^{-1}\circ v_i|\,\leq\, C (|v_i|+1)$ which then implies by Vitali's Convergence Theorem that $G^{-1}(v_i)\to G^{-1}(v)$ in $L^{1}(U)$. Thus using the following Lemma \ref{lem:tight} we can conclude that the distributions of $\ue$ are tight on  $C\big( [0,T], L^1(U) \big)$ as well.

\vskip3ex

In particular, there exists a decreasing sequence $\varepsilon_k \downarrow 0$ such that the distributions of $\ue$ converge weakly to a limiting measure on $C\big( [0,T], L^1(U) \big)$. We may now use Skorohod's observation that we can find a subsequence $\varepsilon_k \downarrow 0$ such that the random functions $u_{\varepsilon_k}$ can be realized on a single probability space $(\tilde{\Omega}, \tilde{\F}, \tilde{\Prob})$. On this space the $u_{\varepsilon_k}$ converge almost surely in $C^0([0,T];L^1(U))$ towards $u$ (see \cite[page 9]{IW89} . By \eqref{eq:energy2} and Fatou's Lemma we can for almost all $\omega\in\Omega, t\in (0,T)$ select a subsequence $\eps_k'\to 0$ such that $\sup_{\eps_k'} E_{\eps_k'} (u_{\eps_k'}(t)) < \infty$. Thus using the Gamma convergence of the functionals $\En$ we can conclude that $u_{\eps_k'}(t)$ converges for a subsequence strongly in $L^1(U)$ to a limit $v\in BV (U; \{-1,1 \})$ and that
\begin{gather}
  \|v\|\,\leq\,\liminf_{\eps_k'\to 0} E_{\eps_k'}(u_{\eps_k'}(t))\label{eq:estGamma}
\end{gather}
But since $u_{\eps_k}\to u$ almost surely in $C([0,T],L^1(U))$  we have $v=u(t,\cdot)$. This proves that
$u(t,\cdot)\in BV(U;\{-1,1\})$; \eqref{eq:energy2}, \eqref{eq:estGamma} yield
\begin{gather*}
  \Ex\Big[ \sup_{0\leq t\leq T} \|u(t)\|_{BV(U)}^p\Big] \,<\,\infty.
\end{gather*}
which concludes the proof.
\end{proof}

\begin{lemma}\label{lem:tight}
Let $\big((X_t^{\eps}, t \in [0,T]) , \eps > 0)$ be a family of stochastic processes taking values in a separable metric space $E$. Let $\tilde{E}$ be another separable metric space and $F\colon E \to \tilde{E}$ a continuous function. Suppose the family of distributions of $X_t^\eps$ is tight on $C([0,T],E)$. Then the family of distributions of $F(X_t^\eps)$ is tight on $C([0,T],\tilde{E})$.
\end{lemma}
\begin{proof}
According to \cite[Thm. 7.2 and Rem. 7.3 on page 128]{EK86} the distributions of $X_t^\eps$ are tight if and only if
\begin{enumerate}
\item  For all $\eta > 0$ there exists a compact set $\Gamma_\eta \subseteq E$ such that for all $\eps$
\begin{gather}
\Prob\Big[ X_t^\eps \in \Gamma_\eta, \quad \forall 0 \leq t \leq T \Big] \, \geq \, 1 - \eta \label{eq:cond1}
\end{gather} 
\item For all $\eta > 0$ there exists $\delta > 0$ such that
\begin{gather}
\sup_t \Prob \Big[\omega\big( X_t^\eps, \delta \big) \geq \eta  \Big] \, \leq \, \eta, \label{eq:cond2}
\end{gather}
\end{enumerate}
where the modulus of continuity is defined as $\omega\big(X_t, \delta \big) = \sup_{t-\delta \leq s \leq t+ \delta} d(X_t,X_s)$.

It is clear from the continuity of $F$ that the $F(X_t^\eps)$ satisfy the property corresponding to \eqref{eq:cond1} if $X_t$ does. 

To see that $F(X_t)$ also satisfies \eqref{eq:cond2} fix $\eta>0$ and the set $\Gamma_{\eta/2}$ such that
\begin{gather*}
\Prob\Big[ X_t^\eps \in \Gamma_{\eta/2}, \quad \forall 0 \leq t \leq T \Big] \, \geq \, 1 - \eta/2. \end{gather*} 
As $\Gamma_{\eta/2}$ is compact $F$ is uniformly continuous when restricted to $\Gamma_{\eta/2}$. Thus one can choose $\eta'$ such that $d(x,y) \leq \eta'$ implies $d(F(x),F(x)) \, \leq \eta$ for all $x,y \in \Gamma_{\eta/2}$. Thus choosing $\delta$ small enough such that $\sup_t \Prob \Big[\omega\big( X_t, \delta \big) \geq \eta'  \Big] \, \leq \, \eta/2$ one obtains
\begin{align}
\sup_t \Prob \Big[\omega\big(F( X_t), \delta \big) \geq \eta  \Big] \, \leq \,& \sup_t \Prob \Big[X_t \notin \Gamma_{\eta/2}  \Big] +  \sup_t \Prob \Big[\omega\big( X_t, \delta \big) \geq \eta'  \Big]  \notag\\
\leq \, & \frac{\eta}{2} + \frac{\eta}{2}.
\end{align}
This finishes the proof.
\end{proof}

\section{Formal sharp interface limit}
In this section we discuss the sharp interface limit $\eps\to 0$ of the stochastic Allen--Cahn equation \eqref{eq:sac}. We will only provide formal arguments and do not intend to give a rigorous justification of a stochastically perturbed mean curvature flow and of the convergence of \eqref{eq:sac} to such a stochastic evolution. Such an analysis is quite delicate and beyond the  aim of this paper.

Our  approach is based on a formal sharp interface limit in a localized version of Proposition \eqref{prop:d-energy-calc} (see \eqref{eq:d-loc-energy} below). For the (deterministic) Allen-Cahn equation such an approach has been successfully used to prove convergence to mean curvature flow (see \cite{Ilma93,MuRoe11}). 

%
We start  with the formulation of a stochastically driven mean curvature flow that formally corresponds to our stochastic Allen-Cahn equation. In the case of a Brownian vector field of the form \eqref{eq_SDEBM} this could be formulated as a stochastic evolution of hypersurfaces $(\Gamma_t)_{t\in (0,T)}$ that is (locally) given by a parametrization $\phi: M\times (0,T)\to\Rn$, $M\subset \Rn$ a smooth $(n-1)$-dimensional reference manifold, such that 
\begin{align}
	d\phi(t,x)\,&=\, \vec{H}(\phi(t,x))\,dt - X^k(t,\phi(t,x))\circ dB^k(t) -X^0(t,\phi(t,x))\circ dt , \label{eq:smcf}
\end{align}
or in It{\^o} formulation
\begin{align}
	d\phi(t,x)\,&=\, \vec{H}(\phi(t,x))\,dt - X^k(t,\phi(t,x)) dB^k(t)  - X^0(t,\phi(t,x)) dt \notag\\
	 &\qquad +\frac{1}{2}DX^k(t,\phi(t,x))X^k(t,\phi(t,x))\,dt. \label{eq:smcf-ito}
\end{align}
As explained above we make a connection between \eqref{eq:sac} and \eqref{eq:smcf} by formally passing to the limit in an equality for the the time-derivative of a \emph{localized} energy and to derive a Brakke-type (in)equality in the limit. 

As above in \eqref{eq:def-mu} we  consider for any $\eta\in C^2(\overline{U})$ and $u\in W^{1,2}(U)$ 
\begin{align}
	\mu_\eps^t(\eta)\,:=\, \int_U \Big(\frac{\eps}{2}|\nabla u_\eps(t)|^2(x) + \frac{1}{\eps}F(u_\eps(t)(x))\Big)\eta(x)\,dx. \label{eq:def-loc-energy}
\end{align}
We then obtain by similar calculations as in the proof of Proposition \ref{prop:d-energy-calc} for a solution $\ue$ of \eqref{eq:sac}
\begin{align}
  \mu_\eps^{t_1}(\eta)&  - \mu_\eps^{t_0}(\eta)   \notag\\
  &= - \int_{t_0}^{t_1} \int_U \eta \frac{1}{\eps} w_\eps^2  \,dx  \, dt +  \int_U\int_{t_0}^{t_1}  \eta w_\eps \, \nabla \ue \cdot X(dt) \,  dx\notag\\ 
&- \frac12  \int_{t_0}^{t_1} \int_U \eta c \cdot w_\eps\nabla \ue \, dx \, dt   \notag\\
&+ \frac 12 \int_{t_0}^{t_1} \mu_\eps^t \Big( \eta \partial_j \Big[  \partial_{x_i} \tilde{a}_{ij} -  \partial_{y_i} \tilde{a}_{ij} \Big] \, \,\Big) \, dt\notag\\
& + \frac12 \int_{t_0}^{t_1} \int_U  \eta \eps \,  \partial_i \ue \,  \partial_{x_k} \partial_{y_k}  \tilde{a}_{ij} \, \partial_j \ue \, dx \, dt\notag\\
&  + \int_{t_0}^{t_1}  \int_U \eta \eps\,  \partial_i \ue  \, \partial_k \big[ \partial_{x_j} \tilde{a}_{ij}  \big]\,  \partial_k  \ue \, dx \, dt \notag\\
&-  \int_{t_0}^{t_1} \int_U \eta \eps\,   \partial_i \ue  \, \partial_j \big[ \partial_{x_k} \tilde{a}_{ij} \big] \, \partial_k  \ue \, dx \, dt\notag \\
	&+ \int_{t_0}^{t_1} \int_U \nabla\eta\cdot w_\eps\nabla \ue\, dx \, dt  -\int_U\int_{t_0}^{t_1}  \eps (\nabla\eta\cdot \nabla\ue)\nabla\ue\cdot X(dt)\,dx\notag \\
	&+  \int_{t_0}^{t_1}  \mu_\eps^t\big(\partial_j\eta\partial_{x_i}\tilde{a}_{ij} + \frac{1}{2} A:D^2\eta\big) \, dt\notag \\
	&+  \int_{t_0}^{t_1} \int_U  \frac{\eps}{2}\nabla\eta\cdot\nabla\ue (c\cdot\nabla\ue) - \eps \partial_j\eta\partial_i\ue \partial_{x_k} \tilde{a}_{ij}\partial_k \ue\,dx\,dt. \label{eq:d-loc-energy}
\end{align}
For the special choice \eqref{eq_SDEBM} of Brownian vector fields this equality yields
\begin{align}
  & \mu_\eps^{t_1}(\eta)  - \mu_\eps^{t_0}(\eta)   \notag\\
  =\,& - \int_{t_0}^{t_1} \int_U\eta\frac{1}{\eps} w_\eps^2  \,dx  \, dt +  \int_U\int_{t_0}^{t_1}  \eta w_\eps \, \nabla \ue \cdot \big(  X^k\,dB_k(t) + X_0 \, dt \big) \,  dx\notag\\
\,&- \frac12  \int_{t_0}^{t_1} \int_U \eta  DX^kX^k \cdot \nabla \ue w_\eps\, dx \, dt   \,+ \notag\\
\,&+ \frac 12 \int_{t_0}^{t_1} \mu_\eps^t \Big( \eta(\nabla\cdot X^k)^2 - \eta\tr (DX^k DX^k) \,\Big) \, dt\notag\\
\,& + \int_{t_0}^{t_1} \int_U  \frac{\eps}{2}\eta|DX^T\nabla\ue|^2 + \eta \eps\nabla\ue\cdot DX^kDX^k\nabla\ue \, dx \, dt\notag\\
\,& -  \int_{t_0}^{t_1} \int_U \eta\eps  (\nabla\cdot  X^k)\nabla\ue\cdot DX^k\nabla\ue \, dx \, dt  \notag\\
	&+ \int_{t_0}^{t_1} \int_U \nabla\eta\cdot w_\eps\nabla \ue\, dx \, dt \notag \\
	&-\int_U\int_{t_0}^{t_1}  \eps (\nabla\eta\cdot \nabla\ue)\nabla\ue\cdot  \big( X^k \,dB_k(t) + X^0 \, dt \big)\,dx \notag \\
	&+  \int_{t_0}^{t_1}  \mu_\eps^t\big(\nabla\eta\cdot X^k \nabla\cdot X^k + \frac{1}{2} X^k\cdot D^2\eta X^k\big) \, dt  \notag \\
	&+  \int_{t_0}^{t_1} \int_U  \frac{\eps}{2}\nabla\eta\cdot\nabla\ue (DX^k X^k \cdot\nabla\ue) - \eps \nabla\eta\cdot X^k \nabla\ue\cdot DX^k\nabla\ue\,dx\,dt. \label{eq:loc-energy-spca}
\end{align}
In an analogous way we consider the derivative of the localized surface area energy
\begin{align}
	\ar^\eta(\Gamma) \,&=\, \int_\Gamma \eta(x) \,d\Ha^{n-1}(x), \label{eq:loc-energy-sharp}
\end{align}
for $\eta\in C^2(\overline{U})$ and a smooth hypersurface $\Gamma\subset \Rn$.
For the first and second variation of the localized energy \eqref{eq:loc-energy-sharp} in the direction of a vector field $Y\in C^2_c(\Rn,\Rn)$ we obtain by using \cite{Simo83}
\begin{align}
	\delta \ar^\eta(\Gamma)(Y) \,=\, &\int_\Gamma \Big(\nabla \eta^\perp\cdot Y - \eta \vec{H}\cdot Y\Big) \,d\Ha^{n-1} \label{eq:fst-var}\\
	\delta^2 \ar^\eta(\Gamma)(Y,Y) \,=\, &\int_\Gamma \eta\Big( (\dive_{\tan} Y)^2 + \sum_{i=1}^{n-1} |D_{\tau_i}^\perp Y|^2 - \sum_{i,j=1}^{n-1}(\tau_i\cdot D_{\tau_j}Y)(\tau_j\cdot D_{\tau_i}Y)\Big)\,d\Ha^{n-1} \notag\\
	&+\int_\Gamma \Big(2(\nabla\eta\cdot Y)(\dive_{\tan} Y) + Y\cdot D^2\eta Y \Big)\,d\Ha^{n-1}, \label{eq:scnd-var}
\end{align}
where for $x\in \Gamma$ we have chosen an arbitrary orthonormal basis $\tau_i(x),\dots,\tau_{n-1}(x)$ of $T_x\Gamma$ and where $\perp$ denotes the normal part of the corresponding vectors.\\
A formal application of It{\^o} formula yields therefore for the formal stochastic evolution \eqref{eq:smcf}
\begin{align}
	d\ar^\eta(\Gamma_t) \,&=\, \int_{\Gamma_t} -\eta \vec{H}\cdot \Big(\vec{H}\,dt - X^k dB^k(t) -X^0 dt +\frac{1}{2}DX^kX^k\,dt\Big)\,d\Ha^{n-1} \notag\\
	&+\int_{\Gamma_t} \nabla\eta^\perp \cdot \Big(\vec{H}\,dt - X^k dB^k(t) -X^0 dt +\frac{1}{2}DX^kX^k\,dt\Big)\,d\Ha^{n-1} \notag\\
	&+ \frac{1}{2} \int_{\Gamma_t} \eta \Big( (\nabla_{\tan} \cdot X^k)^2 + |D_{\tau_i}^\perp X^k|^2 - \sum_{i,j}(\tau_iD_{\tau_j}X^k)(\tau_jD_{\tau_i}X^k)\Big)\,d\Ha^{n-1} \, dt\notag \\
	& +  \int_{\Gamma_t} (\nabla_{\tan}\cdot X^k)(\nabla \eta\cdot X^k) \,d\Ha^{n-1}\, dt\notag \\
	& +\frac{1}{2} \int_{\Gamma_t} X^k\cdot D^2\eta X^k\,d\Ha^{n-1}\, dt. \label{eq:d-area-sharp}
\end{align}
This has to be compared with the sharp interface limit of \eqref{eq:loc-energy-spca}. We assume that the limit of a sequence $(u_\eps)_{\eps>0}$ is given by a smooth evolution $(\Gamma_t)_{t\in (0,T)}$ as above, that for all $t\in (0,T)$ the following convergence properties hold
\begin{align*}
	\mu_\eps^t,\,\eps |\nabla\ue(t,\cdot)|^2\,\LL^n \,&\to c_0\Ha^{n-1}\lfloor \Gamma_t,
	\\
	\eps \nu_\eps(t,\cdot)\otimes\nu_\eps(t,\cdot)\,\LL^n&\to\, c_0\nu(t,\cdot)\otimes\nu(t,\cdot)\Ha^{n-1}\lfloor \Gamma_t,\\
	\frac{w_\eps \nu_\eps}{\eps|\nabla u_\eps|}(t,\cdot)\,\LL^n &\to\, c_0 \vec{H}(t,\cdot)\Ha^{n-1}\lfloor \Gamma_t,\\
	\frac{1}{\eps}w_\eps^2(t,\cdot)\,\LL^n &\to\, c_0 {H}^2(t,\cdot) \Ha^{n-1}\lfloor \Gamma_t,
\end{align*}
where to $x\in \Gamma_t$ we have associated a unit normal vector $\nu(t,x)$. We refer  for example to \cite{MuRoe11} for a precise formulation and rigorous justification of such convergence properties under specific assumptions.\\
 We then obtain from \eqref{eq:loc-energy-spca}
\begin{align}
  & \frac{1}{c_0} \lim_{\eps\to 0} d\mu_\eps^{t}(\eta)  \notag\\
  =\,& -  \int_{\Gamma_t} \eta |\vec{H}|^2  \,d\Ha^{n-1}\, dt +  \int_{\Gamma(t)}  \eta \vec{H} \cdot \big( X^k\,dB_k(t)+ X^0 dt \big) \,  d\Ha^{n-1}\notag\\ 
\,&- \frac12   \int_{\Gamma(t)} \eta  DX^kX^k \cdot \vec{H}\, d\Ha^{n-1}\, dt   \notag\\
\,&+ \frac 12  \int_{\Gamma(t)} \eta(\nabla\cdot X^k)^2 - \eta\tr (DX^k DX^k) \,d\Ha^{n-1}\, dt\notag\\
\,& +   \int_{\Gamma(t)}  \eta|DX^T\nu|^2 +\frac12 \eta \nu\cdot DX^kDX^k\nu \, d\Ha^{n-1}\, dt\notag\\
\,& -   \int_{\Gamma(t)} \eta  (\nabla\cdot  X^k)\nu\cdot DX^k\nu\, d\Ha^{n-1}\, dt\notag\\
	&+  \int_{\Gamma(t)} \nabla\eta\cdot \vec{H}\, d\Ha^{n-1}\, dt \notag\\ 
	&-\int_{\Gamma(t)}  (\nabla\eta\cdot \nu)\nu\cdot X^k \,d\Ha^{n-1}\,dB_k(t) -\int_{\Gamma(t)}  (\nabla\eta\cdot \nu)\nu\cdot X^0 \,d\Ha^{n-1}\,dt\notag \\
	&+    \int_{\Gamma(t)}\big(\nabla\eta\cdot X^k \nabla\cdot X^k + \frac{1}{2} X^k\cdot D^2\eta X^k\big) \, d\Ha^{n-1}\, dt\notag \\
	&+   \int_{\Gamma(t)}\Big(\frac{1}{2} \nabla\eta\cdot\nu (DX^k X^k \cdot\nu) -  \nabla\eta\cdot X^k \nu\cdot DX^k\nu\Big)\,d\Ha^{n-1}\, dt. \label{eq:loc-energy-spca-limit}
\end{align}
Comparing this equation with \eqref{eq:d-area-sharp} and using the calculations from \cite{Le10} we obtain 
\begin{align}
  \frac{1}{c_0} \lim_{\eps\to 0} d \mu_\eps^{t}(\eta)\,&
  =\, d\ar^\eta(\Gamma_t) + \int_{\Gamma(t)} \eta (\nu\cdot X^k  \nu)^2\,d\Ha^{n-1}\, dt. \label{eq:diff-brakke}
\end{align}
We therefore obtain in the formal sharp interface limit of the localized energy equality one additional positive term compared to the corresponding energy derivative for \eqref{eq:smcf}. However, this is consistent with recent work by Le \cite{Le10} that compares the formal sharp interface limit of the second inner variation of $E_\eps$ to the second variation of the area functional and identifies exactly the same extra contribution.\\ 
In this respect the discrepancy should not be interpreted as an indication that \eqref{eq:sac} fails to approximate the stochastic mean curvature flow \eqref{eq:smcf}, but rather as consistent and a consequence of using an energy based proof. In contrast to the deterministic case,  by  It{\^o}'s formula the second inner variation of $E_\eps$ enters in the energy derivative and necessarily produces this extra term.

\def\cprime{$'$}

\end{document}